﻿
\documentclass[twoside,11pt,reqno]{amsart}
\usepackage{amsmath,amssymb,amscd,mathrsfs,amscd, todonotes,appendix}
\usepackage{graphics,verbatim}
\usepackage{todonotes}
\usepackage{enumitem}
\usepackage{hyperref}
\usepackage{setspace} 
﻿
﻿
﻿

﻿
﻿
﻿
﻿

\newcommand{\losemi}{{\otimes \kern -.78em \ltimes}}
\newcommand{\rosemi}{{\otimes \kern -.78em \rtimes}}

\newcommand{\Hom}{\ensuremath{\operatorname{Hom}}}
\newcommand{\End}{\ensuremath{\operatorname{End}}}

\newcommand{\Rad}{\ensuremath{\operatorname{Rad} }}

\newcommand{\Ext}{\operatorname{Ext}}

﻿

﻿
﻿

﻿
\renewcommand{\a}{\alpha}

\newcommand{\si}{\sigma}

﻿

\newcommand{\la}{\lambda}

﻿

﻿
﻿

\newcommand{\Mod}{\operatorname{Mod}}

\newcommand{\St}{\operatorname{St}}

﻿

﻿

\newcommand{\hQ}{\widehat{Q}}

﻿

﻿
\newcommand{\fmn}{{\mathcal F}^m_n}

\newcommand{\gmn}{{\mathcal G}^m_n}

﻿
﻿

﻿
\makeatletter
\newcommand{\leqnomode}{\tagsleft@true}
\newcommand{\reqnomode}{\tagsleft@false}
\makeatother
﻿
﻿

﻿
﻿

﻿
\newtheorem{theorem}{Theorem}[subsection]
﻿

\makeatletter\let\c@fact\c@theorem\makeatother
﻿

\makeatletter\let\c@note\c@theorem\makeatother
﻿

\makeatletter\let\c@lemma\c@theorem\makeatother
﻿

\makeatletter\let\c@lemma\c@theorem\makeatother
﻿

\makeatletter\let\c@alg\c@theorem\makeatother
﻿
\newtheorem{prop}{Proposition}[subsection]
\makeatletter\let\c@prop\c@theorem\makeatother
﻿

﻿
\newtheorem{conj}{Conjecture}[subsection]
\makeatletter\let\c@conj\c@theorem\makeatother
﻿

﻿

\makeatletter\let\c@cor\c@theorem\makeatother
﻿

\makeatletter\let\c@defn\c@theorem\makeatother
﻿
﻿
\theoremstyle{definition}

\makeatletter\let\c@remark\c@theorem\makeatother
﻿
﻿
﻿
﻿
﻿
\makeatletter\let\c@example\c@theorem\makeatother
\numberwithin{equation}{subsection}
﻿
\usepackage[capitalise]{cleveref}
\crefname{theorem}{Theorem}{Theorems}
\crefname{fact}{Fact}{Facts}
\crefname{note}{Note}{Notes}
\crefname{lemma}{Lemma}{Lemmas}
\crefname{alg}{Algorithm}{Algorithms}
\crefname{remark}{Remark}{Remarks}
\crefname{example}{Example}{Examples}
\crefname{prop}{Proposition}{Propositions}
\crefname{conj}{Conjecture}{Conjectures}
\crefname{cor}{Corollary}{Corollaries}
\crefname{defn}{Definition}{Definitions}
\crefname{equation}{\!\!}{\!\!} 
﻿

﻿
﻿
﻿
\newcounter{listequation}

﻿
﻿
\begin{document}
﻿
\title{On Donkin's Tilting Module Conjecture IV: New Frontiers}
﻿
﻿
\author{\sc Christopher P. Bendel}
\address
{Department of Mathematics, Statistics and Computer Science\\
University of
Wisconsin-Stout Polytechnic\\
Menomonie\\ WI~54751, USA}
\thanks{Research of the first author was supported in part by an AMS-Simons Research Enhancement Grant for PUI Faculty}
\email{bendelc@uwstout.edu}
﻿
﻿
﻿
\author{\sc Daniel K. Nakano}
\address
{Department of Mathematics\\ University of Georgia \\
Athens\\ GA~30602, USA}
\thanks{Research of the second author was supported in part by
NSF grant DMS-2401184}
\email{nakano@math.uga.edu}
﻿
﻿
﻿
\author{\sc Cornelius Pillen}
\address{Department of Mathematics and Statistics \\ University
of South
Alabama\\
Mobile\\ AL~36688, USA}
\email{pillen@southalabama.edu}
﻿
\author{Paul Sobaje}
\address{Department of Mathematical Sciences \\
          Georgia Southern University\\
          Statesboro, GA~30458, USA}
\email{psobaje@georgiasouthern.edu}
﻿
\begin{abstract}
Let $G$ be a simple simply connected algebraic group scheme defined over ${\mathbb F}_{p}$, and $G_{r}$ be the $r$th Frobenius kernel. 
Donkin's famous Tilting Module Conjecture purports that a given indecomposable injective $G_{r}$-module can be realized as the restriction of a specific  
tilting module. The conjecture was first stated in 1990 and withstood proof for nearly 30 years. The authors discovered a counterexample in 2019. 
﻿
This paper provides some indication about the rich mathematics surrounding the Tilting Module Conjecture that also discusses an older conjecture by 
Humphreys and Verma. New versions of the Tilting Module Conjectures are formulated given multiple families of new counterexamples. Finally, through a new calculation, 
the authors present further evidence that the Tilting Module Conjecture should hold for reductive algebraic groups with underlying root system of Type $\rm{A}_{n}$ for all fields of characteristic $p>0$.  
\end{abstract}
﻿
﻿
﻿
﻿
\maketitle
﻿
\section{Introduction}
﻿
\subsection{} Let $G$ be a semisimple simply connected algebraic group (scheme) defined over ${\mathbb F}_{p}$, and $k$ be an algebraically closed field of characteristic $p>0$. 
In 1990 at MSRI, Donkin \cite{Don93} proposed his famous Tilting Module Conjecture (TMC) which states that $G$-structures on injective indecomposable modules for
the Frobenius kernels $G_{r}$ should arise from tilting modules for $G$. More precisely, 
﻿
\begin{conj}\label{C:TMC} For all $\lambda\in X_{r}$, 
$$T((p^{r}-1)\rho+\lambda)|_{G_{r}}\cong\widehat{Q}_{r}((p^{r}-1)\rho+w_{0}\lambda).$$
Equivalently, $T(2(p^{r}-1)\rho+w_{0}\lambda)|_{G_{r}}\cong \widehat{Q}_{r}(\lambda)$ for all $\lambda\in X_{r}$. 
\end{conj}
For a description of the notation, see Section 2. The TMC was shown to hold for $p \geq 2h-2$, and for many years, many mathematicians believed that it should hold for all $p$. 
﻿
In 2019, the authors discovered the first counterexample to the Tilting Module Conjecture \cite{BNPS20}. Subsequently, more counterexamples were found in all root systems other than for $\rm{A}_{n}$ and $\rm{B}_{2}$ \cite{BNPS22b}.
So far, no counterexamples to the TMC have been found in type $\rm{A}_n$. The authors provide some evidence for a positive resolution of the TMC for Type $\rm{A}_n$ in \cite{BNPS26}. 
﻿
The authors used the $G$-structure of $\text{Ext}^{1}_{G_{1}}(L(\lambda),L(\mu))^{(-1)}$ for $p$-restricted weights $\lambda$ and $\mu$ in low rank cases to disprove the TMC, and then used 
an inductive procedure to construct infinite families of counterexamples for the higher rank situations. The validity of the TMC implies that these cohomology groups must embed in a tilting module. 
Andersen \cite{And84} proved that  $\text{Ext}^{1}_{G_{1}}(L(\lambda),L(\mu))^{(-1)}$ is a tilting module for $p\geq 3h - 3$. Twenty years later, this bound was lowered to $2h-2$ \cite{BNP04}, 
and after another twenty years to $2h-4$ by the authors in \cite{BNPS24}. 
﻿
\subsection{} At the same conference in 1990, Donkin also proposed his $(p,r)$-Filtration Conjecture that was based on a question posed by Jantzen \cite{Jan80}. The $(p,r)$-Filtration Conjecture 
provides a necessary and sufficient condition for a rational $G$-module to admit a good $(p,r)$-filtration (cf. Section \ref{S:notation}). This conjecture has only been verified for 
$G=SL_{2}$ \cite{And01}.  

Major breakthroughs occurred when 
Kildetoft and Nakano, and Sobaje demonstrated fundamental connections between the TMC and good $(p,r)$-filtrations for $G$-modules (cf. \cite{KN15, So18}). 
For a precise description of these interrelationships the reader is referred to \cite[Section 2.2]{BNPS24} for further details.
﻿
Even though this paper focuses on the TMC and other related questions, the reader should be aware that much of the underlying theory is governed by the existence of good $(p,r)$-filtrations for rational $G$-modules. 
﻿
\subsection{} The goal of this paper is to survey recent developments involving the Tilting Module Conjecture. Even though the general conjecture is false, the TMC holds in wide enough generality to formulate new conjectures on its validity. 
﻿
The paper is organized as follows. In the following section (Section~\ref{S:notation}), the standard notion used throughout the paper is presented. In Section~\ref{S:recall}, we recall the major ideas that were developed in 
\cite{BNPS19, BNPS20, BNPS22a, BNPS22b, BNPS24,BNPS26}. This includes (i) methods involving Donkin's $p$-Filtration Conjecture and Jantzen's Question that reduce verifying the TMC to checking a finite number of representation theoretic 
conditions and (ii) the known list of counterexamples to the TMC. In the case of Type $\rm{A}_{n}$, a recent theorem in \cite{BNPS26} is presented that states that  the TMC is equivalent to proving that composition factor multiplicities are either 
$0$ or $1$ in specific modules for an associated Schur algebra. 
﻿
Based on the information in Section~\ref{S:recall}, new conjectures are stated in Section~\ref{S:conjectures}. These include the Humphreys-Verma Conjecture (the validity of which is implied by that of the TMC) and revised versions of Donkin's Tilting Module Conjecture. 
The Humphreys-Verma Conjecture remains open for all primes $p$. In recent work, Donkin and Geranios \cite{DG26} have shown that this conjecture holds when $\Phi=\rm{G}_{2}$ at $p=2$, even though the TMC fails 
to hold in this case. Finally, in the last section (Section~\ref{S:newcase}), a new case of the TMC is verified when $\Phi=\rm{A}_{4}$ at $p=2$ using in part the methods outlined in Section~\ref{S:recall}.  
﻿
﻿
\subsection{Acknowledgements} This paper is dedicated to our friend and colleague, David J. Benson on occasion of his 70th birthday. Through his books and our conversations with him, Dave was influential in the development of our ideas to utilize cohomological methods in representation theory. 
﻿
We would also like to acknowledge the superb efforts of the organizing committee Srikanth Iyengar, Radha Kessar, Henning Krause, and Julia Petvsova for organizing the Benson conference at the International Centre of Mathematical Sciences in Edinburgh. 
﻿
﻿
﻿
﻿
\section{Preliminaries} \label{S:notation}
﻿
\subsection{Notation} In this paper the standard conventions in \cite{rags} will generally be followed. Throughout this paper $k$ is an algebraically closed field of characteristic $p>0$. Let
\vskip .25cm 
\begin{enumerate}
\item $G$ be a connected semisimple algebraic group scheme defined over ${\mathbb F}_{p}$.
\item $T$ be a fixed split maximal torus in $G$. 
\item $\Phi$ be the root system associated to $(G,T)$. 
\item $\Phi^{\pm}$ be the set of positive (resp. negative) roots. 
\item $\Delta=\{\alpha_1,\dots,\alpha_{l}\}$ be the set of simple roots determined by $\Phi^+$. 
\item $B$ be the Borel subgroup given by the set of negative roots.
\item $U$ be the unipotent radical of $B$. 
\item $\Phi_J$ be the root subsystem of $\Phi$ generated by $J \subseteq \Delta$ with positive subset $\Phi_J^+ = \Phi_J \cap \Phi^+$.
\item $P_J$ be the parabolic subgroup relative to $-J$  with Levi factor $L_J$ and unipotent radical $U_J$.
\item $W$ (resp. $W_J$) be the Weyl group associated with $\Phi$ (resp. $\Phi_J$, i.e., the subgroup of $W$ generated by the reflections associated to the simple roots in $J$).
\item $w_0$ (resp. $w_{J,0}$) denote the longest word of $W$ (resp. $W_J$). 
\item $\rho$ be the half-sum of positive roots (which is also the sum of the fundamental weights). 
\item $\rho_J$ be the half-sum of all the roots spanned by $J$.
\item $\alpha_0$ be the highest short root, with associated coroot $\alpha_0^{\vee}$.
\item $h$ denote the Coxeter number for the root system associated to $G$, i.e., $h = \langle\rho,\alpha_0^{\vee}\rangle + 1$.
\item $X(T)$ be the integral weight lattice spanned by the fundamental weights $\{\omega_1,\dots,\omega_l\}$. 
\item $X^+ := X(T)_{+}$ denote the dominant weights for $G$. 
\item $X_r := X_{r}(T)$ be the $p^{r}$-restricted weights. 
\item $X^+_J$ be the weights in $X(T)$ that are dominant on $J$.
\item $\hat{\lambda}=2(p^{r}-1)\rho+w_{0}\lambda$ for $\la \in X_r(T)$.
\item $\leq$ be the order relation defined on $X$ via $\mu \leq \la$ iff $\la - \mu = \sum_{\a \in \Delta}n_{\a}\a$ for $n_{\a} \in {\mathbb Z}_{\geq 0}$.
\item $\leq_{\mathbb Q}$ be the rationalized order relation with the $n_{\a}$ allowed to lie in ${\mathbb Q}_{\geq 0}$.
\item$\uparrow$ be the strong linkage relation on $X(T)$ (cf. \cite[II.6.4]{rags}).
\item $M^* := \Hom_k(M,k)$ be the linear dual of a rational $G$-module $M$.  
\item $^{\tau}M$ be the vector space $M^*$ with $G$-action twisted by the anti-automorphism $\tau$ (that sends positive root subgroups to negative); cf. \cite[II.2.12]{rags}.
\end{enumerate} 
\vskip .25cm 
\noindent 
Let $F^{r}:G\rightarrow G$ be the $r$th iteration of the  Frobenius morphism, and $G_{r}$ be the scheme theoretic kernel of this map which is often called the 
$r$th Frobenius kernel. 
﻿
\subsection{} There are four fundamental families of finite-dimensional highest weight rational $G$-modules parametrized by the dominant weights $X^+$. For $\lambda\in X^+$, 
one has 
\vskip .15cm 
$\bullet$ $L(\lambda)$ (simple), 
\vskip .15cm
$\bullet$ $\nabla(\lambda)$ (costandard/induced), 
\vskip .15cm 
$\bullet$
$\Delta(\lambda)$ (standard/Weyl), 
\vskip .15cm 
$\bullet$ 
$T(\lambda)$ (indecomposable tilting).
\vskip .15cm\noindent
Note that under $\tau$-duality, $^{\tau}\nabla(\la) = \Delta(\la)$, while $L(\la)$ and $T(\la)$ are self-$\tau$-dual (cf. \cite[II.2.12]{rags}).
﻿
For the Frobenius kernels, $G_{r}$, the simple modules and injective hulls of the simples are parametrized by the $p^{r}$-restricted weights $X_{r}$. For $\lambda\in X_{r}$, one has 
\vskip .15cm 
$\bullet$ $L_r(\lambda)$ (simple), 
\vskip .15cm
$\bullet$ $Q_{r}(\lambda)$ (injective hull of $L_{r}(\lambda)$) and $\widehat{Q}_{r}(\la)$ (injective hull of $L_r(\la)$ as a $G_rT$-module). 
\vskip .25cm 
﻿
For $J \subseteq \Delta$, replacing $G$ by the Levi factor $L_J$, one has ``$J$-versions'' of the fundamental modules as above, denoted for example by $T_J(\la)$, that will be used in Section \ref{S:newcase}.
﻿
Curtis \cite{C60} proved that, for $\lambda\in X_{r}$, the restriction of $L(\lambda)$ to $G_{r}$ is isomorphic to $L_{r}(\lambda)$, thus showing that simple $G_{r}$-modules lift to $G$. Moreover, Steinberg \cite{St63} showed that every simple $G$-module can be written as a twisted tensor product of simple $G_{1}$-modules. Let  $\text{St}_r = L((p^r-1)\rho)$ be the $r$th Steinberg module which is simple as a $G$-module and $G_{r}$-module. Moreover, $\text{St}_{r}$ is  also projective/injective when restricted to $G_{r}$.  
﻿
Lastly, we recall various types of filtrations on a finite-dimensional rational $G$-module $M$.  Suppose there exists a filtation
$$
0 = M_0 \subseteq M_1 \subseteq M_2 \subseteq \cdots \subseteq M_{s-1} \subseteq M_s = M,
$$
such that, for $1 \leq i \leq s$, either $M_i/M_{i-1}$ is zero or
\begin{itemize}
\item $M_i/M_{i-1} \cong \nabla(\mu_i)$ for $\mu_i \in X^+$. Then $M$ has a good filtration.
\item $M_i/M_{i-1} \cong \Delta(\mu_i)$ for $\mu_i \in X^+$. Then $M$ has a Weyl filtration
\item $M_i/M_{i-1} \cong L(\mu_i)\otimes\nabla(\la_i)^{(r)}$ for $\mu_i \in X_r$ and $\la_i \in X^+$.   Then $M$ has a good $(p,r)$-filtration or simply a good $p$-filtration in the case $r = 1$.
\end{itemize}
﻿
﻿
﻿
\section{Recollections} \label{S:recall}
﻿
\subsection{Method for proving the TMC} \label{SS:methods} The Tilting Module Conjecture was elusive for many years due to a lack of effective methods for verifying the conjecture. 
Significant breakthoughs were made by Kildetoft and Nakano \cite{KN15} and later by Sobaje \cite{So18}, when it was discovered that the TMC was related to the existence of good $p$-filtrations. The theorem below (see \cite[Section 4]{BNPS22a}),
which is a culmination of the aforementioned work, provides an highly effective method using representation theoretic conditions for verifying the TMC. 
﻿
\begin{theorem}\label{T:TMCcriterion} Let $\lambda\in X_r$. Suppose that, for all $\sigma \in X_r$, 
\begin{itemize}
\item[(a)] $\St_r \otimes L(\sigma)$ is tilting,  
\item[(b)] $\nabla(\hat{\sigma})$ has  a good $(p,r)$-filtration. 
\end{itemize} 
Then $T(2(p^{r}-1)\rho+w_{0}\lambda)|_{G_{r}}\cong \widehat{Q}_{r}(\lambda)$. 
\end{theorem}
﻿
Prior to our work, the TMC was only known to hold for all primes $p$ in the cases when the underlying root system is of type $\rm{A}_{1}$ or $\rm{A}_{2}$. Theorem \ref{T:TMCcriterion} can be used to check all the rank $2$ cases (except type $\rm{G}_2$ and $p = 2$) and to verify the TMC for type $\rm{A}_{3}$. 
The most challenging case in rank $2$ was the situation where $\Phi=G_{2}$ and $p=7$ which involved the use of validity of the Lusztig Conjecture. 
﻿
\begin{theorem}\label{T:MainTheorem} Let $G$ be a simple, simply connected algebraic group scheme defined and split over ${\mathbb F}_{p}$ and $\Phi$ be its associated root system. The Tilting Module Conjecture 
holds if 
\begin{itemize} 
\item[(a)] $\Phi=\rm{A}_{n}$, $n\leq 3$;
\item[(b)] $\Phi=\rm{B}_{2}$;
\item[(c)] $\Phi=\rm{G}_{2}$ as long as $p\neq 2$.
\end{itemize} 
\end{theorem} 
﻿
\subsection{Counterexamples} \label{SS:counterexamples} In \cite{BNPS22b}, the authors constructed infinitely families of counterexamples to the TMC. Note that one has at least one counterexample 
in every root system other than when $\Phi=A_{n}$ or $B_{2}$. The first counterexample that appeared was for $\Phi=\rm{G}_{2}$ when $p=2$ (cf. \cite{BNPS20}). 
﻿
\begin{theorem}\label{T:counter} Let $G$ be a simple algebraic group over an algebraically closed field of characteristic $p>0$. Assume that the underlying root system $\Phi$ is not of type 
$\rm{A}_{n}$ or $\rm{B}_{2}$. Then there exists a prime $p$ for which $G$ produces a counterexample to the Tilting Module Conjecture. More specifically, 
there exists counterexamples to the TMC for the following groups. 
\begin{itemize} 
\item $\Phi=\rm{B}_{n}$, $p=2$, $n\geq 3$;
\item $\Phi=\rm{C}_{n}$, $p=3$, $n\geq 3$;
\item $\Phi=\rm{D}_{n}$, $p=2$, $n\geq 4$;
\item $\Phi=\rm{E}_{n}$, $p=2$, $n=6,7,8$;
\item $\Phi=\rm{F}_{4}$, $p=2,3$; 
\item $\Phi=\rm{G}_{2}$, $p=2$. 
\end{itemize} 
\end{theorem} 
﻿
﻿
\subsection{Type A}\label{SS:TypeA} In this section, we present some rationale for why the TMC might hold in type $\rm{A}$ for all primes. The main idea is that the representation theory in type $\rm{A}$, i.e, that of $SL_n$, is closely related to that of $GL_n$ for which there is a well-known connection with the representation theory for Schur algebras and symmetric groups. In \cite{BNPS26}, this framework was used to recast the TMC via information about composition factors that involve inverse Schur functors between Schur algebras. 
﻿
For positive integers $n$ and $d$, recall that the Schur algebra is $S(n,d) := \End_{k\Sigma_d}(V^{\otimes d})$, where $V$ is the natural $n$-dimensional $GL_n$-module and $\Sigma_d$ denotes the symmetric group on $d$ objects.  For a positive integer $m \geq n$, there exists an idempotent $f \in S(m,d)$, such that $fS(m,d)f \cong S(n,d)$. Multiplication by $f$ defines an exact functor ${\mathcal F}_{n}^{m}$ from $S(m,d)$-modules to  $S(n,d)$-modules.  More precisely, for any $S(m,d)$-module $M$,
$$
\fmn(M) := fM \cong \Hom_{S(m,d)}(S(m,d)f,M) \cong fS(m,d)\otimes_{S(m,d)}M.
$$
The functor $\fmn$ admits a right adjoint $\gmn$ from $S(n,d)$-modules to $S(m,d)$-modules defined by 
$$
\gmn(M) := \Hom_{S(n,d)}(fS(m,d),M).
$$
The simple $S(n,d)$-modules are parametrized by the set $\Lambda^+(n,d)$ of partitions of $d$ with at most $n$ parts.  For a partition $\la \in \Lambda^+(n,d)$, let $L_n(\la)$ denote the associated simple $S(n,d)$-module.  For more details about these objects and maps, the reader is referred to \cite[Section 2.2]{BNPS26}. 
﻿
With the above notation, we have the following reformulation of the TMC for $SL_n$ (cf. \cite[Corollary 5.3.3]{BNPS26}. Here, $M_{p}$ denotes the Mullineux involution on partitions and $\mu'$ denotes the conjugate or transpose of a partition $\mu$.
﻿
\begin{theorem}\label{T:TMCequivG} Assume $G = SL_n$.  The following are equivalent. 
\begin{itemize} 
\item[(a)] Donkin's Tilting Module Conjecture
\item[(b)] For all $\la \in X_1$ with $\hat{\la} \in \Lambda^+(n,d)$ (when considered as a partition) and $m$ equal to the number of parts in $M_p(\hat{\la})'$, the composition factor multiplicity (as an $S(m,d)$-module)
$$
\left[\gmn(L_n(\si)) : L_m(M_p(\hat{\la})')\right] = 
\begin{cases}
0 \text{ if } \si \neq \la,\\
1 \text{ if } \si = \la,
\end{cases}
$$
for all $\si \in \Lambda^+(n,d)$.
\end{itemize} 
\end{theorem}
﻿
The basis for Theorem \ref{T:TMCequivG} was understanding the socles of tilting modules over a Schur algebra $S(n,d)$.   The following result (cf. \cite[Corollary 4.4.3]{BNPS26}) was a precursor to Theorem \ref{T:TMCequivG} and will be used in the proof of Proposition \ref{P:Filtration}. Recall that a partition is $p$-regular if it does not admit $p$ consecutive nonzero parts that are the same.  

\begin{prop}\label{P:socle} Let $\mu \in \Lambda^+(n,d)$ be a $p$-regular partition and assme that $M_p(\mu)'$ has at most $n$ parts. Then the $S(n,d)$-socle of $T(\mu)$ is $L(M_p(\mu)')$.
\end{prop}
﻿
﻿
\section{Conjectures} \label{S:conjectures}
﻿
\subsection{Humphreys-Verma Conjecture}\label{SS:HVC} Let $G$ be a simple simply connected algebraic group scheme defined over the prime field ${\mathbb F}_{p}$. Curtis \cite{C60} showed that the irreducible 
representations for $G_{1}$ lift to $G$. One can extend this statement for $G_{r}$ and the lifting is unique. 
﻿
In 1973, Humphreys and Verma \cite{HV73} announced, without proof, a theorem that would have implied that the projective indecomposable $G_{r}$-modules also lift to $G$.  The proof that they had expected to work was later found to contain a gap, but over time the implication of their statement became known as the Humphreys-Verma Conjecture (HVC): 
﻿
\begin{conj} Let $G$ be a simple simply connected algebraic group scheme defined over ${\mathbb F}_{p}$.  For each $\la \in X_r$, there exists a finite-dimensional $G$-module $M(\lambda)$ such that $M(\lambda)|_{G_{r}}\cong Q_{r}(\lambda)$. 
\end{conj} 
﻿
Interestingly, one can show that the TMC is equivalent to the theorem stated in \cite{HV73}, so in a sense it, rather than the HVC, truly reflects what Humphreys and Verma had hoped would be true.  In any case, with the conventions being as they are, the TMC is stronger than the so-called HVC, and in fact is now known to be strictly stronger.
﻿
Most importantly, the HVC is still open for all $p$, and may well hold in full generality.  In very recent work, the verification of the HVC for all rank $2$ root systems was completed by Donkin and Geranios \cite{DG26}.  The only rank $2$ case where the TMC fails to hold is when $\Phi = \rm{G}_2$ and $p = 2$, and then only for the injective hull of the trivial module, $Q_1(0,0)$.  Donkin and Geranios were able to show that $Q_{1}(0,0)$ can be realized as a $G$-subquotient of the tilting module $T(2,2)$. 
﻿
﻿
\subsection{Special Root Systems} Given the results of Section \ref{S:recall} (and Section \ref{S:newcase} below), there is a large amount of evidence for the following conjecture. 
﻿
\begin{conj} Let $G$ be a simple simply connected algebraic group scheme defined over ${\mathbb F}_{p}$, and $\Phi$ be its associated root system. Then 
the Tilting Module Conjecture holds for all $p$ if and only if 
\begin{itemize} 
\item[(a)] $\Phi=\rm{A}_{n}$; 
\item[(b)] $\Phi=\rm{B}_{2}$. 
\end{itemize} 
\end{conj} 
﻿
As we have seen, the TMC holds for $\Phi=\rm{B}_{2}$ for all primes. So in order to prove the conjecture, one needs to verify the conjecture when $\Phi=\rm{A}_{n}$ for all primes $p$. 
An equivalent formulation using Schur functors was given in Theorem~\ref{T:TMCequivG}. 
﻿
\subsection{Jantzen's Question and the Humphreys-Verma Conjecture} In the 1980s, Jantzen raised the question of whether every $\nabla(\mu)$ admits a good $p$-filtration for every $\mu \in X^+$.    While the answer is now known to be ``no'' in general for small primes (cf. \cite{BNPS20, BNPS22b}), as evidenced by condition (b) of Theorem \ref{T:TMCcriterion}, there is a close connection between an affirmative answer to Jantzen's Question and the validity of the TMC.  
﻿
In Section \ref{SS:HVC}, we observed that a careful reading of the paper by Humphreys and Verma \cite{HV73} shows that the TMC was a natural updating of the HVC, even though the TMC is a stronger conjecture.  In this subsection we give a second reason that the TMC was a natural successor, namely that whenever Jantzen's Question has an affirmative answer, then the TMC is in that case equivalent to the HVC.
﻿
\begin{theorem}
Let $G$ be a simple simply connected algebraic group scheme defined over ${\mathbb F}_{p}$.  If Jantzen's Question has an affirmative answer for $G$, then the Humphreys-Verma Conjecture for $G$ is equivalent to the Tilting Module Conjecture for $G$.
\end{theorem}
﻿
\begin{proof}
As the TMC always implies the HVC, it is left to show the reverse implication under this hypothesis.  It suffices to assume that $r=1$.  
﻿
Suppose that $\lambda \in X_1$, and that $M$ is a $G$-module such that $M|_{G_1} \cong Q_1(\lambda)$.  Let $\mu \in X^+$.  Because $\Ext_{G_1}^1(M,\nabla(\mu))=0$, the five term exact sequence arising from the Lyndon-Hochschild-Serre spectral sequence for $G_1 \unlhd G$ gives an isomorphism
\[
\Ext_{G/G_1}^1(k,\Hom_{G_1}(M,\nabla(\mu))) \cong \Ext_G^1(M,\nabla(\mu)).
\]
As the functor $\Hom_{G_1}(M,-)$ is exact, if $\nabla(\mu)$ has a good $p$-filtration, then $\Hom_{G_1}(M,\nabla(\mu))$ has a good filtration as a $G/G_1$-module.  From this it follows that
\[
\Ext_{G/G_1}^1(k,\Hom_{G_1}(M,\nabla(\mu))) = 0,
\]
so that  $\Ext_G^1(M,\nabla(\mu)) = 0$.  Thus, $M$ has a Weyl filtration.  This argument shows that any $G$-lift of $Q_1(\lambda)$ has a Weyl filtration, and in particular that ${^{\tau}M}$ does.  Consequently,  $M$ is tilting, so we must have $M \cong T(2(p-1)\rho+w_0\lambda)$.
\end{proof}
﻿
﻿
\subsection{New Bounds for the TMC} Finally, we address how far the general bound on the prime $p$ depending on the root system can be lowered. For nearly 30 years, the TMC was known to hold for $p\geq 2h-2$. Recently, the authors proved that this 
bound can be lowered (cf. \cite[Theorem 1.2.1]{BNPS24}). 
﻿
\begin{theorem}\label{T:2h-4} Let $G$ be a simple algebraic group over an algebraically closed field of characteristic $p>0$ and $h$ be the Coxeter number associated to the root system for $G$. Then the following hold for $p\geq 2h-4$: 
\begin{itemize} 
\item[(a)]  Donkin's Tilting Module Conjecture, 
\item[(b)]  The Humpheys-Verma Conjecture.
\end{itemize} 
\end{theorem} 
﻿
From Theorem~\ref{T:counter}, the TMC can fail even when the prime is good. Furthermore, all the known counterexamples occur when $p<h$. We pose the following intriguing conjecture as a starting point to lower the bound on $p$. 
﻿
\begin{conj} Let $G$ be a simple simply connected algebraic group scheme defined over ${\mathbb F}_{p}$. Then 
the Tilting Module Conjecture holds for $p\geq h$. 
\end{conj} 
﻿
﻿
﻿
﻿
\section{New Type A Calculation} \label{S:newcase}
﻿
\subsection{Type $\rm{A}_4$} As noted in Theorem \ref{T:MainTheorem}, the TMC holds for types $\rm{A}_1$, $\rm{A}_2$, and $\rm{A}_3$. For type $\rm{A}_n$ with $n\geq 4$, from Theorem \ref{T:2h-4}, we know it holds if $p \geq 2n - 2=2h-4$.  In particular, for type $\rm{A}_4$ (i.e., the group $SL_5$), the TMC holds for $p \geq 6$.  That is, the only primes in question are 2, 3, and 5.    In this section, we successfully address the case $p = 2$.  
﻿
\begin{theorem}\label{T:A4} Donkin's Tilting Module Conjecture holds for $SL_5$ (i.e., type $\rm{A}_4$) and the prime $p=2.$
\end{theorem}
﻿
Recall that to verify the TMC, it suffices to consider the $r = 1$ case (cf. \cite[Proposition 2.2.2]{BNPS22a}). 
﻿
﻿
\subsection{Reducing to a Levi}
The following facts were first observed by Donkin   in \cite[Proposition 2.7]{Don93} and  \cite[Proposition 1.5 (ii)]{Don93}, respectively. For more details the reader is referred to \cite[Sections 2.5, 2.6]{BNPS22b}.
﻿
Let $L_J$ be a Levi subgroup of G associated to $J \subseteq \Delta$ and  $\la \in X_r.$ Then one has an equality of $(L_J)_rT$-modules.
﻿
\begin{equation}\label{E:QLevi}
\hQ_{J,r}((p^r-1)\rho + w_{J,0}\la) = \bigoplus_{\nu \in \mathbb{N}J} \widehat{Q}_r((p^r-1)\rho+w_0 \la)_{(p^r-1)\rho+\la-\nu}.
\end{equation}
The  indecomposable tilting modules behave nicely when restricted to Levi subgroups. More precisely,  for any $\la \in X^+$ one obtains:
\begin{equation}\label{E:TLevi}
T_J( \la)=\bigoplus_{\nu \in \mathbb{N}J} T( \la)_{\la-\nu}.
\end{equation} 
Following Sobaje in \cite{So20}, define for $\la \in X_1$
$$t(\la) =\text{ch} (T((p-1)\rho+\la))/\chi((p-1)\rho),$$
$$q(\la)= \text{ch} (\hQ_1((p-1)\rho+w_0 \la))/\chi((p-1)\rho).$$
For $\la \in (X_{J})_1$, set
$$t_J(\la) =\text{ch} (T_J((p-1)\rho_J+\la))/\chi((p-1)\rho_J),$$
$$q_J(\la)= \text{ch} (\hQ_{J,1}((p-1)\rho_J+w_{J,0} \la))/\chi((p-1)\rho_J).$$
For $\mu \in X^+$, let $s(\mu) = \sum_{w \in W}e(w\mu)$ and $s_J(\mu) = \sum_{w \in W_J}e(w\mu),$which allows us to define $a_{\mu}^{\la},$ and $b_{\mu}^{\la}$ via 
\begin{equation}
q(\la) = \sum_{\mu \in X^+} a_{\mu}^{\la} s(\mu) \text{ and } t(\la) = \sum_{\mu \in X^+} b_{\mu}^{\la} s(\mu).
\end{equation}
Similarly, set 
\begin{equation}
q_J(\la) = \sum_{\{\mu \in X^+ | \la-\mu \in \mathbb{N}J\}} (a_{J})_{\mu}^{\la} \;s_J(\mu) \text{ and } t_J(\la) = \sum_{\{\mu \in X^+ | \la-\mu \in \mathbb{N}J\}}  (b_{J})_{\mu}^{\la} \; s_J(\mu).
\end{equation}
It follows now from (\ref{E:QLevi}) and (\ref{E:TLevi}):
﻿
\begin{prop}\label{RedLevi}
If the TMC holds for $L_J$ and $\la - \mu \in \mathbb{N}J$, then
$$a_{\mu}^{\la}=(a_{J})_{\mu}^{\la} \;=(b_{J})_{\mu}^{\la} \;=b_{\mu}^{\la}.$$
\end{prop} 
﻿
﻿
\subsection{Criterion for Summands}
The following criterion assures that certain injective $G_1$-modules do not appear as $G_1$-summands of a tilting module. This may be viewed as a sort of inductive version of Theorem \ref{T:TMCcriterion}.
﻿
\begin{prop}\label{P:Summand}
Let $\la, \mu \in X_1$ with $\mu <_{\mathbb{Q}} \la.$
Assume the following. 
\begin{itemize}
\item[(a)] For all $\gamma \in X_1,$ $ \St_1 \otimes L(\gamma)$ has a good filtration.
\item[(b)] $T((p-1)\rho + \mu)|_{G_1} = \hQ_1((p-1)\rho+w_0\mu).$
\item[(c)] $\nabla((p-1)\rho + \mu)$ has a good $p$-filtration.
\end{itemize}
Then $\hQ_1((p-1)\rho+w_0\mu)$ is not a $G_1$-summand of $T((p-1)\rho + \la).$
\end{prop}
﻿
\begin{proof}
By \cite[Theorem 4.3.2]{BNPS22a}, it follows from (a), (b), and (c) that the canonical projection
$$\St_1 \otimes T((p-1)\rho + \mu) \twoheadrightarrow \St_1 \otimes L((p-1)\rho+w_0\mu)$$ splits. 
From the argument in the proof of \cite[Proposition 4.2.2]{BNPS22a}, this implies that $\St_1 \otimes \Rad(T((p-1)\rho + \mu))$ has a good filtration. Hence, its dual, $\St_1 \otimes T((p-1)\rho + \mu)/L((p-1)\rho+w_0\mu)$, has a Weyl-filtration. Moreover, (a) implies that $M=\Hom_{G_1}(L((p-1)\rho+w_0\mu), T((p-1)\rho+\la))$ has a Weyl-filtration.
It follows that\\ \\
$\Ext_G^1(T((p-1)\rho + \mu)/L((p-1)\rho+w_0\mu)\otimes M, T((p-1)\rho+\la))$
\begin{eqnarray*} &\hookrightarrow&\Ext_G^1(T((p-1)\rho + \mu)/L((p-1)\rho+w_0\mu)\otimes M, \St_1 \otimes T(\la))\\
&\cong & \Ext_G^1(\St_1 \otimes T((p-1)\rho + \mu)/L((p-1)\rho+w_0\mu) \otimes M,  T(\la))=0.
\end{eqnarray*}
The assertion now follows from \cite[Lemma 4.2.2]{So18}.
\end{proof}
﻿
\subsection{A good $2$-filtration}
Throughout this (and the next) subsection $G$ is always of type $\rm{A}_4$ and $p=2.$ To prove Theorem \ref{T:A4}, we need the following:
﻿
\begin{prop}\label{P:Filtration}
$\nabla(2,1,1,2)$ has a good $2$-filtration.
\end{prop}
﻿
\begin{proof}
The characters of the simple restricted modules for $\rm{A}_4$ and $p=2$ are known, see for example \cite[Table III.A.5]{DS}. Therefore, the composition factors and their multiplicities for any induced module can be calculated. Relevant to us are the following multiplicities:
\begin{equation} 
[\nabla(2,1,1,2):L(0,2,2,0)]=1,
\end{equation}
\begin{equation} 
[\nabla(2,1,1,2):L(2,1,2,0)]=[\nabla(2,1,1,2):L(0,2,1,2)]=1,
\end{equation}
and 
\begin{equation} 
[\nabla(2,1,2,0):L(0,1,1,0)]=[\nabla(0,2,1,2):L(0,1,1,0)]=1.
\end{equation}
Note that the only weights $\mu = \mu_0 +2 \mu_1$ with $\mu < (2,1,1,2)$ and $\nabla(\mu_1) \neq L(\mu_1)$ are $$\{(0,2,2,0),  (2,1,2,0), (0,2,1,2) \}.$$
One concludes from the above that for  $\nabla(2,1,1,2)$ to have a good $2$-filtration, each of  
$\displaystyle{\nabla(0,1,1,0)^{(1)},  L(0,1,0,0) \otimes \nabla(1,0,1,0)^{(1)} \text{ and } L(0,0,1,0) \otimes \nabla(0,1,0,1)^{(1)}}$ have to appear once as a subquotient.

Now  \cite[Theorem 7.1]{Kop} together with  (5.4.1) and (5.4.2) imply that there exist non-zero homomorphisms
$$\phi_1:\nabla(2,1,1,2) \to \nabla(0,2,2,0),$$
 $$\phi_2: \nabla(2,1,1,2) \to \nabla(2,1,2,0),$$
 $$\phi_3: \nabla(2,1,1,2) \to \nabla(0,2,1,2),$$
 each unique up to multiples.
 
As discussed in Case 2 in the proof of Theorem \ref{T:A4}, the tilting module  $T(2,1,1,2)$ is isomorphic to $\hQ_{1}(0,1,1,0)$ as a $G_1$-module. Note that Proposition \ref{P:Filtration} is not used in the proof of Case 2. Hence, $\nabla(2,1,1,2),$ which is a quotient of this tilting module, has simple head isomorphic to $L(0,1,1,0).$
 The same is true for the the heads of the images of the $\phi_i.$ They are all simple and isomorphic to $L(0,1,1,0).$

We will first look at the image of $\phi_1.$ From the calculations and tables in \cite[III.2.3, III.2.4]{DS} one concludes the following:
If $\mu \uparrow (0,2,2,0)$ and $\Ext^1_G(L(0,2,2,0), L(\mu))\neq 0,$ then $\mu$ is the zero weight and $\dim \Ext^1_G(L(0,2,2,0), k)=1.$ 
This implies that the  $\nabla(0,2,2,0)$ has a unique submodule of length two. It is isomorphic to $\nabla(0,1,1,0)^{(1)}$. Hence, any submodule of $\nabla(0,2,2,0)$ of length greater than or equal to two, including the image of $\phi_1,$ will have $\nabla(0,1,1,0)^{(1)}$ as a submodule. Therefore,  $\nabla(0,1,1,0)^{(1)}$ is a subquotient of $\nabla(2,1,1,2).$

In order to describe the image of $\phi_2$, we will apply techniques developed in \cite{BNPS26}; see the discussion in Section \ref{SS:TypeA}. We may regard the tilting module $T(2,1,2,0)$ over $SL_5$ as a module for the Schur algebra $S(5,10)$. Note that the partition corresponding to the $T$-weight $(2,1,2,0)$ is $\mu := (5,3,2),$ a $2$ regular partition of $10$.  We wish to apply Proposition \ref{P:socle} to identify the $S(5,10)$-socle of $T(\mu)$.   For $p =2$, the Mullineux map is the identity, and so $\si := M_p(\mu)' = \mu' = (5,3,2)' = (3,2,2,1,1)$ has 5 parts.  Hence, Proposition \ref{P:socle} implies that the $S(5,10)$-socle of $T(\mu)$ is $L(\si)$.  The partition $(3,3,2,1,1)$ translates back to the weight $(0,1,1,0)$ over $\rm{A}_4$. In particular, this means the $SL_5$-socle of $T(2,1,2,0)$ is the simple module $L(0,1,1,0)$. 

Since $\nabla(2,1,2,0)$ is a quotient of $T(2,1,2,0)$ one concludes that the head of $\nabla(2,1,2,0)$ is simple and isomorphic to $L(0,1,1,0)$. But by (5.4.3) the multiplicity of $L(0,1,1,0)$ in $\nabla(2,1,2,0)$ is one. Therefore, the map $\phi_2$ is surjective and its image isomorphic to $\nabla(2,1,2,0),$ which has $L(0,1,0,0) \otimes \nabla(1,0,1,0)^{(1)}$ as a submodule. 
One obtains the desired subquotient $L(0,1,0,0) \otimes \nabla(1,0,1,0)^{(1)}$ of $\nabla(2,1,1,2).$

The existence of the remaining subquotient $L(0,0,1,0) \otimes \nabla(0,1,0,1)^{(1)}$ by using $\phi_{3}$ follows from the symmetry of the Dynkin diagram.
\end{proof}
﻿
﻿
\subsection{Proof of the Theorem}
In this subsection, assume that $G$ is of type $\rm{A}_4$ and $p = 2$. To prove Theorem \ref{T:A4}, it suffices to show that, for all $\la \in X_1$ and $\mu \in X^+$, $a_{\mu}^{\la} = b_{\mu}^{\la}.$ Note that $T((p-1)\rho +\la)$ is a $G$-summand of $\St_{1}\otimes L(\la).$ It follows that  $a_{\mu}^{\la} \neq 0$ forces the multiplicity of the $\mu$-weight space in $L(\la),$ denoted by $L(\la)_{\mu},$ to be nonzero. In addition, we know from \cite{So20} that  $a_{\mu}^{\la} \neq 0$ also implies that $\mu-\rho  \uparrow \la- \rho.$ For $\la \in X_1$, set
$$A_{\la}= \{ \mu \in X^+\; | \; \mu \neq \la,\ L(\la)_{\mu} \neq 0, \text{ and } \mu-\rho  \uparrow \la- \rho\}.$$
The following table lists all the cases to be considered in order to show that $a_{\mu}^{\la} = b_{\mu}^{\la}$ always holds.  We omit the dual weights of  $\la$ that appear in Cases 3 and 4.
\vskip.2cm
\begin{tabular}{|l|l||l|l|}
\hline
Case & $\la \in X_1$ & $A_{\la}$ & $\la - \mu$\\
\hline
1 &$\omega_i,$ $1\leq i \leq 4$ & $\emptyset$ & \\
\hline
2 & $\omega_i+ \omega_j,$ $1\leq i<j \leq 4$& $\emptyset$ & \\
\hline
3 & $(1,1,1,0)$ & $\{(0,1,0,1)\}$ &$\alpha_1+\alpha_2+\alpha_3$ \\
\hline
4 &$(1,1,0,1)$ & $\{(0,1,0,0)\}$ &$\alpha_1+\alpha_2+\alpha_3+\alpha_4$ \\
\hline
5 & $(1,1,1,1)$ & $\{(2,0,1,0),$ & $\alpha_2+\alpha_3+\alpha_4$ \\
6& & $(0,1,0,2),$ & $\alpha_1+\alpha_2+\alpha_3$ \\
 7& & $(1,0,0,1)\}$ &$\alpha_1+2\alpha_2+2\alpha_3+\alpha_4$ \\
\hline
\end{tabular}
\vskip.2cm
﻿
Note that $a_{\la}^{\la}=b_{\la}^{\la}=1.$ Hence, there is nothing to show for Cases 1 and 2. For Cases 3, 5, and 6, since the TMC holds over type $\rm{A}_3$, we may apply Proposition \ref{RedLevi} with an appropriate Levi subgroup to obtain the desired result.
﻿
Consider Case 4, where $\la = (1,1,0,1)$ and the goal is to show that $T(2,2,1,2) = \hQ(0,1,0,0)$.   The only case in question is whether $a^{\la}_{\mu} = b^{\la}_{\mu}$ for $\mu = (0,1,0,0)$.  That is, we need to show that $\hQ_1((p-1)\rho +w_0\omega_2)=\hQ_1(1,1,0,1)$ does not appear as a $G_1$-summand of $T((p-1)\rho+ (1,1,0,1)) = T(2,2,1,2).$  Note that by Case 1 (when $\la = \omega_2$), $\hQ(1,1,0,1) \cong T((p-1)\rho + \omega_2) = T(1,2,1,1)$.  Let $\Mod((p-1)\rho + (1,1,0,1)) = \Mod((2,2,1,2))$ be the category of finite-dimensional $G$-modules with composition factors whose highest weights are less than or equal to $ (2,2,1,2)$.  Our goal is to show that $T(1,2,1,1)$ is projective and injective in this category. To that end, we make use of  \cite[Proposition 3.1.1]{BNPS22a} with $\sigma = (1,1,0,1),$ $\la_0=\omega_2,$  and $\la_1=0.$ Note that in this case $- w_0\sigma+ \la_0= \rho.$ The only dominant weights $\gamma$ with $2 \gamma \leq \rho$ are $0$ and $(1,0,0,1).$ Neither weight extends with the trivial weight. It follows from \cite[Proposition 3.1.1(a)]{BNPS22a} that $\St \otimes L(\omega_2)$ is projective  and injective as a $G$-module in $\Mod((2,2,1,2))$. The module $T(1,2,1,1)\cong \hQ_1(1,1,0,1)$ has simple $G$-socle $L(1,1,0,1).$ Since it is a summand of $\St_1 \otimes L(\omega_2)$, it is also projective and injective in $\Mod((2,2,1,2)).$ It can therefore not appear as a $G_1$-summand of $T(2,2,1,2)$, as any copy of $T(1,2,1,1)$ would necessarily split off from $T(2,2,1,2)$.  This implies that $a_{(0,1,0,0)}^{(1,1,0,1)} = b_{(0,1,0,0)}^{(1,1,0,1)}.$
﻿
This leaves Case 7. Note that the only possibility for the TMC to fail now is for $\hQ_1((p-1)\rho +w_0\alpha_0)=\hQ(0,1,1,0)\cong T(2,1,1,2)$ (by Case 2) to appear as a $G_1$-summand of $T(2(p-1)\rho).$ To exclude this, we wish to apply Proposition \ref{P:Summand} with $\lambda = (1,1,1,1)$ and $\mu = (0,1,1,0)$.   Considering the necessary hypotheses as given in Proposition \ref{P:Summand}, Condition (a) is known from \cite[Theorem 5.4.1]{BNPS19}, Condition (b) follows from Case 2 above, and Condition (c) follows from Proposition \ref{P:Filtration}.  Hence, the TMC holds for $\rm{A}_4$ and $p=2$.  
﻿
﻿
﻿
﻿
﻿
\providecommand{\bysame}{\leavevmode\hbox
to3em{\hrulefill}\thinspace}

﻿
\end{document}